\def\timestamp{%
Time-stamp: <dim-and-maps.tex: Wednesday 30-09-2009 at 23:48:57 (cest)>}
\def\stripname Time-stamp: <#1 #2>{#2}
\edef\filedate{\expandafter\stripname\timestamp}
\documentclass[a4paper]
               {amsart}[2004/08/06]
\usepackage[backrefs]{amsrefs}
\newcommand\calB{\mathcal{B}}
\newcommand\calC{\mathcal{C}}
\newcommand\calO{\mathcal{O}}

\newcommand\calQ{\mathcal{Q}}
\newcommand\calS{\mathcal{S}}
\newcommand\calT{\mathcal{T}}
\DeclareMathSymbol\BB0{AMSb}{`B}
\DeclareMathSymbol\N 0{AMSb}{`N}
\DeclareMathSymbol\R 0{AMSb}{`R}
\DeclareMathSymbol\Q 0{AMSb}{`Q}
\newcommand\Nstar{\N^*}
\let\emptyset \undefined
\let\ge       \undefined
\let\le       \undefined
\DeclareMathSymbol\restr\mathbin{AMSa}{"16}
\DeclareMathSymbol\le   \mathrel{AMSa}{"36}    
\DeclareMathSymbol\ge   \mathrel{AMSa}{"3E}    
\DeclareMathSymbol\emptyset\mathord{AMSb}{"3F}
\newcommand\cee{\mathfrak{c}}
\newcommand\CH{\mathsf{CH}}
\newcommand\cl[1]{\operatorname{cl}#1}
\newcommand\Int[1]{\operatorname{int}#1}
\newcommand\diam {\operatorname{diam}}
\newcommand\Fr   {\operatorname{Fr}}
\newcommand\ind  {\operatorname{ind}}
\newcommand\Ind  {\operatorname{Ind}}
\newcommand\ord  {\operatorname{ord}}
\newcommand\RO   {\operatorname{RO}}
\newcommand\preim{^{-1}}
\let\epsilon\varepsilon
\newtheorem{theorem}{Theorem}[section]
\newtheorem{corollary}[theorem]{Corollary}
\newtheorem{lemma}[theorem]{Lemma}
\newtheorem{proposition}[theorem]{Proposition}
\numberwithin{equation}{section}
\newtheorem*{thm}{Theorem}

\theoremstyle{definition}
\newtheorem*{question}{Question}

\theoremstyle{remark}
\newtheorem{remark}[theorem]{Remark}

\begin{document}

\title[Dimension and maps]%
      {Covering dimension and finite-to-one maps}

\dedicatory{To Ken Kunen on the occasion of his retirement from teaching}

\author{Klaas Pieter Hart}

\address{Faculty EEMCS\\TU Delft\\
         Postbus 5031\\2600~GA {} Delft\\the Netherlands}
\address{Department of Mathematics and Statistics\\
         Miami University\\
         Oxford\\
         OH~45056}
\email{k.p.hart@tudelft.nl}
\urladdr{http://fa.its.tudelft.nl/\~{}hart}

\author{Jan van Mill}
\address{Faculty of Sciences\\
         Division of Mathematics\\
         Vrije Universiteit\\
         De Boelelaan 1081\textsuperscript{a}\\
         1081 HV Amsterdam\\
         The Netherlands}
\email{j.van.mill@cs.vu.nl}

\date{\filedate}

\begin{abstract}
Hurewicz' characterized the dimension of separable metrizable spaces
by means of finite-to-one maps. 
We investigate whether this characterization also holds in the class
of compact $F$-spaces of weight~$\cee$.
Our main result is that, assuming the Continuum Hypothesis, an $n$-dimensional
compact $F$-space of weight~$\cee$ is the continuous image of a zero-dimensional
compact Hausdorff space by an at most $2^n$-to-$1$ map.
\end{abstract}



\maketitle

\section*{Introduction}

The starting point for this note is a theorem of Hurewicz 
from~\cite{Hurewicz1926},  which characterizes the dimension of separable
metrizable spaces in terms of maps.

\begin{thm}
Let $X$ be a separable and metrizable space and let $n$ be a natural number.
Then $\dim X\le n$ if{}f there are a zero-dimensional separable metrizable
space~$Y$ and a continuous and closed surjection $f:Y\to X$ such that 
$\bigl|f\preim(x)\bigr|\le n+1$ for all~$x$.
\end{thm}

Our aim is to generalize this result to a wider class of spaces.

One half of Hurewicz's theorem is a special case of the theorem on dimension
raising maps.
This special case can be generalized to the class of normal spaces
(the hint to Problem~1.7.F (c) in \cite{MR1363947} provides a proof):

\begin{thm}
Let $f:Y\to X$ be a closed continuous surjection between normal spaces
and $n$ a natural number.
If $Y$ is strongly zero-dimensional and if $f$ is such that 
$\bigl|f\preim(x)\bigr|\le n+1$ for all~$x$ then $\Ind X\le n$.
\end{thm}

Hurewicz' proof of the other half was based on the interplay between the large 
inductive dimension and the covering dimension, using finite collections of
closed sets of order~$n+1$ to construct the preimage.
Also Kuratowski's quantitative proof in~\cite{Kuratowski1932} used covering
dimension to show that in the case where $X$~has no isolated points
the set of surjections with at most $n+1$-point fibers is residual in the 
space of all surjections from the Cantor set onto~$X$.

This all suggests that we should look for classes of compact Hausdorff
spaces where the covering dimension and the large inductive dimension
coincide.
We shall show that the compact $F$-spaces of weight~$\cee$ form such
a class, assuming the Continuum Hypothesis.
Under that assumption these spaces are one step up from compact metrizable
spaces: the weight is $\aleph_1$ and the $F$-space property lets one do
countably many things at a time, which is quite helpful in inductions
and recursions of length~$\omega_1$.

In Section~\ref{sec.Hurewicz} we give a variation of Hurewicz' proof of the 
second half his theorem; we do this with an eye to the proof of our main
result and to make Hurewicz' argument better known.
Next we establish, in Section~\ref{sec.dim=ind=Ind}, that the three major 
dimension functions coincide on the class of compact $F$-spaces of 
weight~$\cee$.
In Sections~\ref{sec.special.bases} and~\ref{sec.fin-to-one} we show that
every $n$-dimensional compact $F$-space is an at~most $2^n$-to-$1$ continuous
image of a zero-dimensional compact space.
This leaves open an obvious question:

\begin{question}[$\CH$]
Is every $n$-dimensional compact $F$-space of weight~$\cee$ a continuous image
of some zero-dimensional compact space by a map whose fibers have size at 
most~$n+1$?
\end{question}

\subsection*{Some preliminaries}

We follow Engelking's book, \cite{MR1363947}, regarding dimension theory
except that by the order of a family~$\calT$ at a point~$x$ we simply
mean the cardinality of $\{T\in\calT:x\in T\}$.
Thus a (normal) space has covering dimension at most~$n$ if{}f every finite
open cover has an open refinement of order at most $n+1$.

An \emph{$F$-space} is one in which separated $F_\sigma$-sets have 
disjoint closures. 
Whenever $X$ is a $\sigma$-compact, locally compact space that is not compact
then the remainder $\beta X\setminus X$ in its \v{C}ech-Stone compactification
is an $F$-space.
Gillman and Jerison's book, \cite{MR0407579}, is still a good source of basic
information on $F$-spaces.

A set is \emph{regular open} if it is the interior of its closure.
The family $\RO(X)$ of regular open sets in the space~$X$ forms a Boolean
algebra under the following operations:
$U\vee V=\Int\cl(U\cup V)$, $U\wedge V=U\cap V$ and $U'=X\setminus\cl U$.
In the proof of our main result we obtain our zero-dimensional preimage
as the Stone space of a subalgebra of~$\RO(X)$; we refer to 
Koppelberg's book, \cite{MR991565}, for a comprehensive treatment
of this construction.

\section{Making a zero-dimensional preimage}
\label{sec.Hurewicz}

As announced we preent in this section a variation of Hurewicz' construction
of a zero-dimensional preimage of an $n$-dimensional compact metrizable
space by a map whose fibers have cardinality at most~$n+1$.

The key notion is that of a \emph{tiling} of a space, which we define to be 
a finite pairwise disjoint family of regular open sets whose closures form 
a cover of the space.
Given a tiling~$\calT$ and a point~$x$ we put 
$\calT_x=\{T\in\calT:x\in\cl T\}$.

In the proof we will construct ever finer tilings of the space; the following
lemma will help us keep the cardinalities of the families~$\calT_x$ under 
control.

\begin{lemma}\label{lemma.1}
Let $\{B_i:i<k\}$ be a family of regular open sets in a space~$X$
and let $T$ be a regular open set that is covered by this family.
For each~$i$ put $C_i=T\cap\bigl( B_i\setminus\bigcup_{j<i}\cl B_j\bigr)$ 
and for $x\in\cl T$ put $F_x=\{i:x\in \cl C_i\}$.
Then $x\in\Fr B_i$ whenever $i\in F_x$ is not maximal.
\end{lemma}

\begin{proof}
This is clear: if $j<i$ then $B_j\cap \cl C_i=\emptyset$.
Furthermore, if $j<i$ in $F_x$ 
then $x\in\cl C_j\cap\cl C_i\subseteq \cl B_j\cap\cl C_i$;
so $x\in\cl B_j\setminus B_j=\Fr B_j$.
\end{proof}

The previous lemma implies that $\{i:x\in\cl C_i\}$ has at most one element
more than $\{i:x\in\Fr B_i\}$.
This behaviour persists when we refine tilings.

\begin{lemma}\label{lemma.2}
Let $\calT$ be a tiling of $X$ that is a subset of the Boolean algebra 
generated by a finite family~$\calB$ of regular open sets and assume that
for every~$x$ the family~$\{B\in\calB:x\in\Fr B\}$ has cardinality 
at least $|\calT_x|-1$.
Fix one~$T\in\calT$ and a finite family~$\calC=\{C_i:i<k\}$ of regular open 
sets such that $T\subseteq\bigcup\calC$ and $\calB\cap\calC=\emptyset$; 
for $i<k$ put $R_i=T\cap\bigl(C_i\setminus\bigcup_{j<i}\cl C_j\bigr)$.

Then the tiling $\calS=\calT\setminus\{T\}\cup\{R_i:i<k\}$
has the same property as~$\calT$:
for every~$x$ the family $\{B\in\calB\cup\calC:x\in\Fr B\}$ has
cardinality at least $|\calS_x|-1$.
\end{lemma}

\begin{proof}
Let $x\in X$.
If $x\notin\cl T$ then $\calS_x=\calT_x$ and we are done by the assumption.

Consider the case $x\in\cl T$.
Then $\calS_x$ consists of $\calT_x\setminus\{T\}$ and 
$\{R_i:i\in G_x\}$, where $G_x=\{i:x\in\cl R_i\}$; this implies
that $|\calS_x|=|\calT_x|-1+|G_x|$.

We can apply Lemma~\ref{lemma.1} to see that $x\in\Fr C_i$ if 
$x\in\cl R_i$ and $i\neq\max G_x$.
This implies that $\{C\in\calC:x\in\Fr C\}$ has cardinality at least~$|G_x|-1$.
By assumption the family $\{B\in\calB:x\in\Fr B\}$ has cardinality at 
least $|\calT_x|-1$.

Because $\calB$ and $\calB'$ are disjoint the family 
$\{B\in\calB\cup\calC:x\in\Fr B\}$ has cardinality at least 
$\bigl(|\calT_x|-1)+\bigl(|G_x|-1)$, which is equal to~$|\calS_x|-1$. 
\end{proof}

Now we are ready to reprove the second half of Hurewicz' theorem.
              
\begin{theorem}
Let $X$ be a metrizable compact space with $\dim X\le n$.
Then there are a zero-dimensional compact metrizable space~$Y$ and
a continuous surjection $f:Y\to X$ such that 
$\bigl|f\preim(x)\bigr|\le n+1$ for all~$x$.
\end{theorem}

\begin{proof}
Let $\{B_i:i<\omega\}$ be a base for $X$ consisting of regular open sets
and such that $\bigcap_{i\in F}\Fr B_i=\emptyset$ whenever 
$F\in[\omega]^{n+1}$;
see, e.g., \cite{MR1206002}*{Corollary~6.12}. 
For technical reasons we assume that for every isolated point~$x$ in~$X$
the set $\bigl\{i:B_i=\{x\}\bigr\}$ is infinite; 
because $\Fr\{x\}=\emptyset$ whenever $x$~is isolated this does not interfere 
with the intersection property of the family $\{\Fr B_i:i<\omega\}$.

By recursion we construct a sequence $\langle \calT_k:k<\omega\rangle$ of 
tilings such that for all~$k$ the family $\calT_{k+1}$ refines $\calT_k$,
such that $\ord\{\cl T:T\in\calT_k\}\le n+1$ for all~$k$, and such that
$\lim_k\max\{\diam T:T\in\calT_k\}=0$.

Each $\calT_k$ is given the discrete topology and our compact zero-dimensional
space~$Y$ is the subspace of the product $\prod_{k<\omega}\calT_k$ defined by
$$
\langle T_k:k<\omega\rangle\in Y
\text{ iff } (\forall k)(T_{k+1}\subseteq T_k)
$$
Clearly $Y$~is closed in the product, hence compact, metrizable and 
zero-dimensional.
If $\langle T_k:k<\omega\rangle\in Y$ then $\bigcap_{k<\omega}\cl T_k$ consists
of exactly one point (by compactness and the diameter condition);
this defines a map~$f$ from~$Y$ to~$X$.

The map is (uniformly) continuous: if $\epsilon>0$ there is a~$k$ such that
$\diam T<\epsilon$ for all~$T\in\calT_k$.
If $T=\langle T_k:k<\omega\rangle$ and $S=\langle S_k:k<\omega\rangle$ in~$Y$ 
are such that $T_k=S_k$ then $f(T),f(S)\in\cl T_k$ and 
hence $d\bigl(f(T),f(S)\bigr)<\epsilon$.

The map~$f$ is onto: if $x\in X$ then it is easy to find $T\in Y$ such that
$x\in T_k$ for all~$k$; then $x=f(T)$.

The map~$f$ is at most $n+1$-to-one.
Indeed, let $x\in X$ and for each~$k$ let 
$\calT_{k,x}=\{T\in\calT_k:x\in\cl T\}$.
Then $|\calT_{k,x}|\le|\calT_{k+1,x}|\le n+1$ for all~$k$; this means
that there is a~$k_0$ such that $|\calT_{k,x}|=|\calT_{k_0,x}|$ for $k\ge k_0$.
And this implies that $T\mapsto T_{k_0}$ is a bijective map from $f\preim(x)$
to~$\calT_{k_0,x}$ and thus: $\bigl|f\preim(x)\bigr|\le n+1$.

It remains to construct the $\calT_k$.
We set~$\calT_0=\{X\}$.
We assume we have found $\calT_k$ as a subset of the Boolean algebra generated
by $\{B_i:i< l\}$ for some~$l$ and that the assumptions of Lemma~\ref{lemma.2}
are satisfied: for every $x\in X$ there are at least $|\calT_{k,x}|-1$ 
indices~$i<l$ such that $x\in\Fr B_i$.

Because for each isolated point~$x$ the set~$\{x\}$ occurs as $B_i$ 
infinitely often we know that for every~$m$ the family $\{B_i:i\ge m\}$ 
is a base for~$X$.
We can therefore find pairwise disjoint finite subsets $F_T$ of
$(l,\omega)$ for $T\in\calT_k$ such that $\cl T\subseteq\bigcup_{i\in F_T}B_i$
and $\diam B_i<2^{-k}$ for all~$i$.
We can use these, as in Lemma~\ref{lemma.2} to define tilings of 
each~$T\in\calT_k$: for $i\in F_T$ put 
$C_{T,i}=T\cap\bigl(B_i\setminus\bigcup_{j\in F_T,j<i}\cl B_j\bigr)$.
Repeated application of Lemma~\ref{lemma.2} shows that
$\calT_{k+1}=\{C_{T,i}:i\in F_T,T\in\calT_k\}$ has the same property 
as~$\calT_k$: for each~$x$ there are at least $|\calT_{k+1,x}|-1$ indices~$i$ 
in $l\cup\bigcup_{T\in\calT_k}F_T$ such that $x\in\Fr B_i$.

To see that $|\calT_{k,x}|\le n+1$ for all~$k$ and all~$x$ we combine
two inequalities:
first, by construction we have 
$\bigl|\{i:x\in \Fr B_i\}\bigr|\ge |\calT_{k,x}|-1$;
second, by assumption on our base we have $n\ge\bigl|\{i:x\in \Fr B_i\}\bigr|$.
This implies $|\calT_{k,x}|-1\le n$.
\end{proof}

As mentioned in the introduction, in~\cite{Kuratowski1932} Kuratowski gave
a quantitative version of this half of Hurewicz' theorem: if $X$~is compact,
metrizable, $n$-dimensional and without isolated points then the set
of maps all of whose fibers have size $n+1$ or less is residual in the
space of all surjections from the Cantor set to~$X$.
The covering dimension is invoked to show that, given a natural number~$k$,
the set of maps with a fiber of size at least $n+2$ in which the points are
at least distance~$2^{-k}$ apart is nowhere dense (it is also readily seen
to be closed).

\section{Equality of dimensions}
\label{sec.dim=ind=Ind}

It is well known that the three fundamental dimension functions take on the
same values for all separable metrizable spaces.
We prove that this also holds in the class of compact $F$-spaces of 
weight~$\cee$, provided the Continuum Hypothesis holds.

In the proof we use Hemmingsen's characterization of covering dimension
(\cite{MR1363947}*{Theorem~1.6.9}):
$\dim X\le n$ if{}f every $n+2$-element open cover has an open shrinking
with empty intersection.

\begin{theorem}[$\CH$]
Let $X$ be a compact $F$-space of weight~$\cee$.
Then $\dim X=\ind X=\Ind X$.\label{thm.dim=ind=Ind}  
\end{theorem}

\begin{proof}
The inequalities $\dim X\le\ind X\le\Ind X$ are well known.
We show $\Ind X\le\dim X$ by showing that $\dim X\le n$
implies that between any two disjoint closed sets $F$ and~$G$ one can find
a partition~$L$ that satisfies $\dim L\le n-1$.
This is known to be true in case~$n=0$, so we assume $n\ge1$ from now on.  

Fix a base $\calB$ for~$X$ that consists of cozero sets,
has cardinality~$\aleph_1$ (by the $\CH$) and is closed under countable unions
and finite intersections.

Let $\langle\calB_\alpha:\alpha<\omega_1\rangle$ enumerate the family
of all $n+1$-element subfamilies of~$\calB$ with cofinal repetitions.
We write $\calB_\alpha=\{B_{\alpha,i}:i\le n\}$.

We construct, by recursion, two sequences 
$\langle U_\alpha:\alpha<\omega_1\rangle$ and
$\langle V_\alpha:\alpha<\omega_1\rangle$
in~$\calB$ such that
\begin{enumerate}
\item $F\subseteq U_0$ and $G\subseteq V_0$;
\item $\cl U_\alpha\subseteq U_\beta$ and $\cl V_\alpha\subseteq V_\beta$
      whenever $\alpha<\beta$;
\item $\cl U_\alpha\cap\cl V_\alpha=\emptyset$ for all~$\alpha$;
\item if $U_\alpha\cup V_\alpha\cup\bigcup\calB_\alpha=X$ then there is a
 subfamily $\calB'_\alpha=\{B'_{\alpha,i}:i\le n\}$ of~$\calB$ that 
 refines~$\calB_\alpha$ and is such that
 $U_{\alpha+1}\cup V_{\alpha+1}\cup\bigcup\calB'_\alpha=X$ and
 $\bigcap\calB'_\alpha\subseteq U_{\alpha+1}\cup V_{\alpha+1}$. 
\end{enumerate}
Then $L=X\setminus\bigcup_\alpha(U_\alpha\cup V_\alpha)$ is a partition between
$F$ and $G$ and $\dim L\le n-1$.
That $L$~is a partition between~$F$ and~$G$ follows from~(1), (2) and~(3).
To see that $\dim L\le n-1$ let $\calO$ be an $n+1$-element basic open cover 
of~$L$.
There is an~$\alpha$ such that $\calO=\{L\cap B_{\alpha,i}:i\le n\}$ and
such that $X\setminus(U_\alpha\cup V_\alpha)\subseteq\bigcup\calB_\alpha$.
By construction $\calO'=\{L\cap B'_{\alpha,i}:i\le n\}$ is a refinement
of~$\calO$ such that~$\bigcap\calO'=\emptyset$.

It remains to perform the construction.
We obtain $U_0$ and $V_0$ using compactness and the fact that $\calB$ is closed
under finite unions.
If $\alpha$ is a limit we let $U_\alpha=\bigcup_{\beta<\alpha}U_\beta$
and $V_\alpha=\bigcup_{\beta<\alpha}V_\beta$.
Then $U_\alpha, V_\alpha\in\calB$ by the assumption that $\calB$ is closed
under countable unions and $\cl U_\alpha\cap\cl V_\alpha=\emptyset$
because $X$~is an $F$-space.

To deal with the successor case we first take elements 
$C$ and~$D$ of~$\calB$ such that
$\cl U_\alpha\subseteq C$, 
\ $\cl V_\alpha\subseteq D$ and $\cl C\cap\cl D=\emptyset$.

If $U_\alpha\cup V_\alpha\cup\bigcup\calB_\alpha=X$ we 
put $E_i=B_{\alpha,i}\setminus(\cl U_\alpha\cup\cl V_\alpha)$ for $i\le n$.
Then we apply the inequality $\dim X\le n$ to find a shrinking
$\{B'_{\alpha,i}:i\le n\}\cup\{O\}$ of
$\{E_i:i\le n\}\cup\{C\cup D\}$ such that 
$\cl O\cap\bigcap_i\cl B'_{\alpha,i}=\emptyset$.
Let $O_1=O\cap C$ and $O_2=O\cap D$; also let $K=\bigcap_i\cl B'_{\alpha,i}$.
Note that $\cl U_\alpha\subseteq O_1$ and $\cl V_\alpha\subseteq O_2$
and that the family $\{K,\cl O_1, \cl O_2\}$ is pairwise disjoint.
We choose $U_{\alpha+1}$ and $V_{\alpha+1}$ in~$\calB$ with disjoint closures
such that $K\cup \cl O_1\subseteq U_{\alpha+1}$ and
$\cl O_1\subseteq U_{\alpha+1}$.
Then the conclusion of~(4) is satisfied.

If $U_\alpha\cup V_\alpha\cup\bigcup\calB_\alpha\neq X$ we let
$U_{\alpha+1}=C$ and $V_{\alpha+1}=D$.
\end{proof}

\begin{remark}
This proof is similar to the one given in~\cite{MR1363947} of the analogous 
result for compact metrizable spaces.
That proof used a metric to guide the countably many steps toward a partition
of covering dimension at most $n-1$.
Of course in an infinite compact $F$-space there is no metric available; in our
proof the role of the metric is taken over (in the background) by the unique
uniformity that generates the topology of the compact $F$-space.
\end{remark}

\begin{remark}
The second-named author has constructed an example of a compact $F$-space
of weight~$\cee^+$ with non-coinciding dimensions, \cite{vanMill-inprep}.  
\end{remark}

\section{Special bases}
\label{sec.special.bases}

In Section~\ref{sec.Hurewicz} we used the fact that a metrizable compact
space~$X$ with $\dim X\le n$ has a base $\{B_i:i<\omega\}$ with the property
that $\bigcap_{i\in F}\Fr B_i=\emptyset$ whenever $|F|=n+1$.
This is a special case of a stronger structural statement: every metrizable 
compact space has a base $\{B_i:i<\omega\}$ with the property that
$\dim\bigcap_{i\in F}\Fr B_i\le\dim\Fr B_{i_0}-|F|+1$, where $i_0=\min F$.

Our goal is to prove a similar statement for compact $F$-spaces of 
weight~$\cee$, assuming the Continuum Hypothesis.

In general, if $X$ is a compact space of weight~$\aleph_1$ we shall assume
it is embedded in the Tychonoff cube~$[0,1]^{\omega_1}$ and for 
$\alpha<\omega_1$ we write $X_\alpha=\{x\restr\alpha:x\in X\}$; this is the 
projection of~$X$ onto the first $\alpha$ coordinates.
We denote this projection map by~$p_\alpha$, we reserve $\pi_\alpha$ for the
projection onto the $\alpha$th coordinate.

\begin{lemma}
There is a closed and unbounded subset~$C$ of~$\omega_1$ such that
$\dim X_\alpha=\dim X$ for $\alpha\in C$.  
\end{lemma}

\begin{proof}
The cube $[0,1]^{\omega_1}$ has a nice subbase, which consists of
the strips $\pi_\alpha\preim\bigl[[0,q)\bigr]$ and 
$\pi_\alpha\preim\bigl[(q,1]\bigr]$, where $\alpha<\omega_1$ and 
$q\in\Q\cap(0,1)$.
We close this subbase under finite unions and intersections to obtain a 
base~$\calB$ for the cube.

First we assume $\dim X=n<\infty$.
In this case if $\calB'$ is a finite subfamily of~$\calB$ that covers~$X$ 
then there is another finite subfamily~$\calB''$ of~$\calB$ that also 
covers~$X$, refines~$\calB'$ and is such that 
$\bigl|\{B\in\calB'':x\in B\}\bigr|\le n+1$ for all~$x\in X$.

Observe that each finite subfamily~$\calC$ of~$\calB$ is supported by a finite
subset~$F_\calC$ of~$\omega_1$ (the coordinates of the strips used to make
its elements).
Next note that, given $\alpha<\omega_1$, there are only countably many finite
subfamilies of~$\calB$ whose support lies below~$\alpha$.
Thus we obtain a function $f:\omega_1\to\omega_1$, defined by
\begin{quote} 
$f(\alpha)$~is the first countable ordinal~$\beta$ such that whenever 
$\calB'$~is a finite subfamily of~$\calB$ that covers~$X$ and whose support 
lies below~$\alpha$ then it has a refinement of order at most~$n+1$ whose 
support lies below~$\beta$.
\end{quote}
The set 
$C=\{\delta<\omega_1:(\forall\alpha)(\alpha<\delta \to f(\alpha)<\delta)\}$
is closed and unbounded and it should be clear that $\dim X_\delta\le n$
whenever $\delta\in C$.
To get equality we note that there is also a finite cover~$\calC$ of~$X$ by 
members of~$\calB$ for which \emph{every} open refinement has order~$n+1$.
Upon deleting an initial segment of~$C$ we can assume that~$\calC$ is supported
below~$\min C$; then $\calC$~witnesses that $\dim X_\delta\ge n$ for 
all~$\delta\in C$.

In case $X$~is infinite-dimensional we have for each~$n$ a finite 
cover~$\calC_n$ such that every open refinement has order at least~$n$.
For any~$\alpha$ above the supports of these covers we 
have $\dim X_\alpha=\infty$.
\end{proof}

The following proposition is instrumental in the construction of the type
of base alluded to above.
It also provides another proof of Theorem~\ref{thm.dim=ind=Ind}.
In it we use the notion of a $P$-set: a subset of a space is a $P$-set
if whenever it is disjoint from an $F_\sigma$-subset it is also disjoint the
closure of thet $F_\sigma$-set; in terms of neighbouirhoods:
the intersection of countably many neighbourhoods of the set is again a 
neighbourhood of the set.
The members of our base will have nowhere dense closed $P$-sets for boundaries.

\begin{proposition}[$\CH$]\label{prop.part}
Let $X$ be a compact $F$-space of weight~$\cee$.
Let $F$ and $G$ be disjoint closed subsets of~$X$ and let $\calQ$ be a family
of no more than~$\aleph_1$ many nowhere dense closed $P$-sets in~$X$.
There are disjoint regular open sets $U$ and $V$ such that
\begin{enumerate}
\item $F\subseteq U$ and $G\subseteq V$,
\item $P=X\setminus(U\cup V)$ is a nowhere dense $P$-set,
\item if $\dim X<\infty$ then $\dim P\le\dim X-1$,
\item if $Q\in\calQ$ then $P\cap Q$ is nowhere dense in~$Q$ and
      if $\dim Q<\infty$ then $\dim(P\cap Q)\le\dim Q-1$.
\end{enumerate}
\end{proposition}

\begin{proof}
We may as well assume that $\dim X=n<\infty$ and that $\calQ$ has 
cardinality~$\aleph_1$. 
The proof is easily modified in case either of these is not the case.

Choose a closed and unbounded set~$C$ such that $\dim X_\delta=\dim X$ 
for $\delta\in C$ and assume without loss of generality (and by compactness)
that $p_\delta[A]\cap p_\delta[B]=\emptyset$, where $\delta=\min C$.

Enumerate $\calQ$ as $\{Q_\alpha:\alpha<\omega_1\}$ and choose for each~$\alpha$
a closed and unbounded subset~$C_\alpha$ of~$C$ such that 
$\dim p_\delta[Q_\alpha]=\dim Q_\alpha$ whenever $\delta\in C_\alpha$.
Because the intersection of countably many closed and unbounded sets
is again closed and unbounded we may as well assume that 
$C_\beta\subseteq C_\alpha$ whenever $\alpha<\beta$.

In case $\delta\in C_\alpha$ we can choose a zero-dimensional
$F_\sigma$-set $Z_{\alpha,\delta}$ in~$X_\delta$ such that
$Z_{\alpha,\delta}$~is dense in~$X_\delta$, the intersection 
$Z_{\alpha,\delta}\cap p_\delta[Q_\alpha]$~is dense in~$p_\delta[Q_\alpha]$, 
and such that
\begin{equation}\label{eq.5}
\dim(X_\delta\setminus Z_{\alpha,\delta})\le n-1\text{ and }
\dim(p_\delta[Q_\alpha]\setminus Z_{\alpha,\delta})\le \dim Q_\alpha-1
\end{equation}

We start our recursive construction of~$P$.
Along the way we construct a sequence 
$\langle\delta_\alpha:\alpha<\omega_1\rangle$ of ordinals.

Let $\delta_0=\min C_0$ and choose a partition~$L$ in~$X_{\delta_0}$ between
$p_{\delta_0}[F]$ and $p_{\delta_0}[G]$ that is disjoint from~$Z_{0,\delta_0}$.
Thus we obtain automatically that
\begin{itemize}
\item $\dim L\le n-1$,
\item $\dim L\cap p_{\delta_0}[Q_0] \le \dim Q_0-1$, and
\item $L$~is nowhere dense in $X_{\delta_0}$ and 
      $L\cap p_{\delta_0}[Q_0]$ is nowhere dense in $p_{\delta_0}[Q_0]$.
\end{itemize}
Write $X_{\delta_0}\setminus L=U\cup V$, where $U$ and $V$ are open and disjoint
sets around $p_{\delta_0}[F]$ and $p_{\delta_0}[G]$ respectively.
Let $V_0=X_\delta\setminus\cl U$ and $U_0=X_\delta\setminus \cl V_0$;
then 
\begin{itemize}
\item $U_0$ and $V_0$ are regular open,
\item $P_0=X_\delta\setminus(U_0\cup V_0)$ is a subset of~$L$ and a partition 
      between $p_{\delta_0}[F]$ and~$p_{\delta_0}[G]$ 
      with $\dim P_0\le\dim L\le n-1$, and
\item $\dim P_0\cap p_{\delta_0}[Q_0]\le 
     \dim L\cap p_{\delta_0}[Q_0] \le \dim p_{\delta_0}[Q_0]-1=\dim Q_0-1$.
\end{itemize}
Observe that $\cl U_0=U_0\cup P_0$ and $\cl V_0=V_0\cup P_0$.

To find $\delta_1$ observe first that $p_{\delta_0}\preim[U_0]$ 
and $p_{\delta_0}\preim[V_0]$ are disjoint open $F_\sigma$-sets of~$X$ and hence
have disjoint closures as $X$ is an $F$-space.
As with $F$ and $G$ we can find an ordinal $\eta$ such that
$p_\eta\bigl[\cl p_{\delta_0}\preim[U_0]\bigr]$ and 
$p_\eta\bigl[\cl p_{\delta_0}\preim[V_0]\bigr]$
are disjoint.
Pick $\delta_1\in C_1$ above~$\eta$.

In $X_{\delta_1}$ we can find a partition $L$ between 
$p_{\delta_1}\bigl[\cl p_{\delta_0}\preim[U_0]\bigr]$ and 
$p_{\delta_1}\bigl[\cl p_{\delta_0}\preim[V_0]\bigr]$
that is disjoint from~$Z_{0,\delta_1}\cup Z_{1,\delta_1}$ ---
this is possible because $Z_{0,\delta_1}\cup Z_{1,\delta_1}$ is zero-dimensional
by the countable closed sum theorem.
We now obtain, by (\ref{eq.5}):
\begin{itemize}
\item $\dim L\le n-1$, 
\item $\dim L\cap p_{\delta_1}[Q_0] \le \dim Q_0-1$, and
\item $\dim L\cap p_{\delta_1}[Q_1] \le \dim Q_1-1$.
\end{itemize}
Because of the density conditions on~$Z_{0,\delta_1}$ and $Z_{1,\delta_1}$ we know
that $L$~is nowhere dense in~$X_{\delta_1}$, that $L\cap p_{\delta_1}[Q_0]$ is 
nowhere dense in~$p_{\delta_1}[Q_0]$ and that $L\cap p_{\delta_1}[Q_1]$ is 
nowhere dense in~$p_{\delta_1}[Q_1]$.

As above we find disjoint regular open sets $U_1$ and $V_1$ around 
$p_{\delta_1}\bigl[\cl p_{\delta_0}\preim[U_0]\bigr]$ and 
$p_{\delta_1}\bigl[\cl p_{\delta_0}\preim[V_0]\bigr]$
respectively such that $P_1=X_{\delta_1}\setminus(U_1\cup V_1)\subseteq L$.
Note also that $p^{\delta_1}_{\delta_0}[P_1]\subseteq P_0$.

At stage $\alpha$ we consider the disjoint open $F_\sigma$-sets
$A=\bigcup_{\beta<\alpha}p_{\delta_\beta}\preim[U_\beta]$ and 
$B=\bigcup_{\beta<\alpha}p_{\delta_\beta}\preim[V_\beta]$.
There is a $\delta_\alpha\in C_\alpha$ above $\{\delta_\beta:\beta<\alpha\}$
such that $p_{\delta_\alpha}[\cl A]$ and $p_{\delta_\alpha}[\cl B]$ are disjoint.

The union $Z=\bigcup_{\beta\le\alpha}Z_{\beta,\delta_\alpha}$ is zero-dimensional
by the countable closed sum theorem so we can find a partition~$L$ 
in~$X_{\delta_\alpha}$ between $p_{\delta_\alpha}[\cl A]$ and $p_{\delta_\alpha}[\cl B]$
that is disjoint from~$Z$.
As before we find
\begin{itemize}
\item $\dim L\le n-1$,
\item $\dim L\cap p_{\delta_\alpha}[Q_\beta]\le\dim Q_\beta-1$ 
      for $\beta\le\alpha$, and
\item $L$ and $L\cap p_{\delta_\alpha}[Q_\beta]$ are nowhere dense 
      in~$X_{\delta_\alpha}$ and $p_{\delta_\alpha}[Q_\beta]$ respectively.
\end{itemize}
As before we find disjoint regular open sets $U_\alpha$ and $V_\alpha$
around $p_{\delta_\alpha}[\cl A]$ and $p_{\delta_\alpha}[\cl B]$ respectively
such that $P_\alpha=X_{\delta_\alpha}\setminus(U_\alpha\cup V_\alpha)$
is a subset of~$L$.

At the end let $U=\bigcup_\alpha p_{\delta_\alpha}\preim[U_\alpha]$
and $V=\bigcup_\alpha p_{\delta_\alpha}\preim[V_\alpha]$.
then $U$ and $V$ are disjoint open sets around $F$ and~$G$ respectively, so that
$P=X\setminus(U\cup V)$ is a partition between $F$ and~$G$.

To see that $P$ is a $P$-set observe that 
$\cl p_{\delta_\beta}\preim[U_\beta]\subseteq p_{\delta_\alpha}\preim[U_\alpha]$
whenever $\beta<\alpha$; by compactness this implies that $\cl K\subseteq U$
whenever $K$~is an $F_\sigma$-subset contained in~$U$.
The same applies for~$V$, so that $P\cap\cl K=\emptyset$, whenever $K$ is 
an $F_\sigma$-set that is disjoint from~$P$.

To see that $\dim P\le n-1$ observe that $p_{\delta_\alpha}[P]\subseteq P_\alpha$
for all~$\alpha$.
Any finite basic open cover of~$P$ is supported below $\delta_\alpha$ for
some~$\alpha$; because $\dim P_\alpha\le n-1$ this cover has an open refinement
of order at most~$n$ that is also supported below~$\delta_\alpha$.

To see that $P$~is nowhere dense let $B$ be any basic open set in 
$[0,1]^{\omega_1}$ that meets~$X$ and choose $\alpha$ such that $B$~is supported
below~$\delta_\alpha$.
As $P_\alpha$ is nowhere dense there is a basic open set~$B'\subseteq B$, 
also supported below~$\delta_\alpha$, that meets~$X_{\delta_\alpha}$ but is
disjoint  from~$P_\alpha$.
Reinterpreted in~$X$ this means that $B'\subseteq B$, that $B'$ meets~$X$
and that $B'\cap P=\emptyset$.

To see that $\dim P\cap Q_\alpha<\dim Q_\alpha$ and $P\cap Q_\alpha$~is nowhere 
dense in~$Q_\alpha$ apply the previous two paragraphs inside the 
space~$Q_\alpha$, both times taking suitable~$\delta$s inside~$C_\alpha$.

In case $\calQ$ is countable one needs only one closed and unbounded set:
the intersection of the countably many associated to~$X$ and the members 
of~$\calQ$.

In case $\dim X=\infty$ one chooses the dense zero-dimensional subsets
$Z_{\alpha,\delta}$ as above, this to make all intersections $P\cap Q_\alpha$
nowhere dense, but one only worries about the value of
$\dim(p_\delta[Q_\alpha]\setminus Z_{\alpha,\delta})$
in case~$\dim Q_\alpha<\infty$.
\end{proof}

In the following theorem we adopt the convention that $\infty-n=\infty$ whenever
$n$~is a natural number --- in this way the statement will be valid both
for finite- and infinite-dimensional spaces.

\begin{theorem}[$\CH$]\label{thm.nice.base}
Let $X$ be a compact $F$-space of weight~$\cee$.
Then $X$~has a base~$\calB=\{B_\alpha:\alpha<\omega_1\}$ such that
$\dim \Fr B_\alpha\le\dim X-1$ for all~$\alpha$ and
$\dim\bigcap_{\alpha\in F}\Fr B_\alpha\le\dim\Fr B_{\min F}-|F|+1$
whenever $F$~is a finite subset of~$\omega_1$.
\end{theorem}

\begin{proof}
Let $\calC$ be a base for~$X$ of cardinality~$\aleph_1$ and let 
$\bigl\{\langle C_\alpha,D_\alpha\rangle:\alpha<\omega_1\bigr\}$ 
enumerate the set of pairs 
$\langle C,D\rangle\in\calC^2$ that satisfy $\cl C\subseteq D$.

Apply Proposition~\ref{prop.part} repeatedly to find, for each~$\alpha$,
disjoint regular open sets $U_\alpha$ and $V_\alpha$ around 
$\cl C_\alpha$ and $X\setminus D_\alpha$ respectively such that 
$P_\alpha=X\setminus(U_\alpha\cup V_\alpha)$ is a nowhere dense $P$-set that 
satisfies
\begin{itemize}
\item $P_\alpha = \Fr U_\alpha$,
\item $\dim P_\alpha\le \dim X-1$,
\item for every finite subset~$F$ of $\alpha$ one has
$\dim (P_\alpha\cap\bigcap_{\beta\in F}P_\beta)\le\dim\bigcap_{\beta\in F}P_\beta-1$.
\end{itemize}
Then $\{U_\alpha:\alpha<\omega_1\}$ is the base that we seek.
\end{proof}

A special case of this theorem is the one that we shall use in the next section.

\begin{theorem}[$\CH$]\label{thm.base.ord.n}
Let $X$ be a compact $F$-space of weight~$\cee$ and of finite dimension~$n$.
Then $X$~has a base~$\calB=\{B_\alpha:\alpha<\omega_1\}$, consisting
of regular open sets, such that
$\dim\bigcap_{\alpha\in F}\Fr B_\alpha=\emptyset$ whenever $F$~is an 
$n+1$-element subset of~$\omega_1$.
\end{theorem}

\begin{proof}
Let $\{B_\alpha:\alpha<\omega_1\}$ be a base as in Theorem~\ref{thm.nice.base}.
Then $\dim\Fr B_\alpha\le n-1$ for all~$\alpha$, so if $|F|=n+1$ then
$\dim\bigcap_{\alpha\in F}\Fr B_\alpha\le n-1-(n+1)+1=-1$,
which means that $\bigcap_{\alpha\in F}\Fr B_\alpha=\emptyset$.
\end{proof}

\section{Finite-to-one maps}
\label{sec.fin-to-one}

The purpose of this section is to show that, assuming the Continuum Hypothesis,
every finite-dimensional compact $F$-space of weight~$\cee$ is a finite-to-one
continuous image of a compact zero-dimensional space of weight~$\cee$.

\begin{theorem}[$\CH$]\label{thm.2n.op.1}
Let $X$ be a compact $F$-space of weight~$\cee$ of finite dimension~$n$.
Then $X$ is the at most $2^n$-to-$1$ continuous image of a compact 
zero-dimensional space of weight~$\cee$.
\end{theorem}

\begin{proof}
Let $\calB=\{B_\alpha:\alpha<\omega_1\}$ be a base for~$X$ as in 
Theorem~\ref{thm.base.ord.n}.
Let $\BB$ be the Boolean subalgebra of $\RO(X)$ generated by this base
and let $Y$ be the Stone space of~$\BB$.
If $y\in Y$ then $\bigcap\{\cl C:C\in y\}$ consists of exactly one point,
which we denote by~$f(y)$.
Let $x\in X$ and put $F=\{\alpha:x\in\Fr B_\alpha\}$.
If $f(y)=x$ then $y$ determines a function $p_y:F\to 2$ by
$p_y(\alpha)=1$ iff $B_\alpha\in y$; in addition if $\alpha\notin F$
then $x\in B_\alpha$ or $x\notin\cl B_\alpha$.
It follows that if $f(y)=f(z)=x$ then $B_\alpha\in y$ iff $B_\alpha\in z$
for $\alpha\notin F$, so if $y\neq z$ then $p_y\neq p_z$.
This implies that $\bigl|f\preim(x)\bigr|\le2^{|F|}\le2^n$.
\end{proof}

\begin{corollary}[$\CH$]
If $X$ is a one-dimensional compact $F$-space of weight~$\cee$ then $X$
is an at most~$2$-to-$1$ continuous image of a compact zero-dimensional
space of weight~$\cee$.  
\end{corollary}

Thus, for compact one-dimensional $F$-spaces we have a direct generalization
of Hurewicz' theorem, as $1+1=2$.

One can give a proof of Theorem~\ref{thm.2n.op.1} along the lines of the
proof in Section~\ref{sec.Hurewicz}.
We take a base as in Theorem~\ref{thm.base.ord.n} but enumerate it in
such a way that every singleton open set is counted cofinally often.

Again one constructs tilings $\calT_\alpha$ of order~$n+1$ but one can only
ensure that $\calT_{\alpha+1}$ refines~$\calT_\alpha$ for every $\alpha$.
The reason becomes apparent at stage~$\omega$: the common refinement of
the tilings~$\calT_m$ will be infinite and not usable as a factor in a
compact product.
What one can do is start a fresh $\omega$-sequence of tilings at each limit 
ordinal~$\lambda$.
The tilings~$\calT_{\lambda+m}$ will be constructed from the family
$\{B_\alpha:\alpha\ge\delta\}$ for some~$\delta$ (depending on~$\lambda$).
The zero-dimensional space~$Y$ will consist of the 
points $\langle T_\alpha:\alpha<\omega_1\rangle\in\prod_\alpha\calT_\alpha$ 
with the following properties:
\begin{itemize}
\item $T_{\alpha+1}\subseteq T_\alpha$ for all~$\alpha$;
\item $\{T_\alpha:\alpha<\omega_1\}$ has the finite intersection property.
\end{itemize}
For each~$x$ there will be at most $n$ limite ordinals $\lambda$ such that
$x$~is on the boundary of a tile in one of the $\calT_{\lambda+k}$
(and hence in $\calT_{\lambda+l}$ for $l\ge k$).
Let $\langle \lambda_i:i<p\rangle$ enumerate these limit ordinals and for 
each~$i$ let $m_i$ be the maximum of $\{|\calT_{\lambda_i+k,x}|:k\in\omega\}$.
The fiber of~$x$ under the obvious map from~$Y$ onto~$X$ has cardinality
$\prod_{i<p}m_i$.
For each~$i$ we get at least $m_i-1$ boundaries that contain~$x$, so that
$n\ge\sum_{i<p}(m_i-1)$.
From this it easy to deduce that $2^n\ge\prod_{i<p}m_i$, so that this map
has fibers of size at most~$2^n$ as well.

\subsection{Universality}

The proofs of Hurewicz and Kuratowski show that if a space is compact, 
metrizable, $n$-dimensional and without isolated points then it is an at 
most~$n+1$-to-$1$ continuous image of the Cantor set.
It is also well known that every compact and metrizable is a continuous image
of the Cantor set, see for example \cite{MR1363947}*{1.3.D}.
Thus the Cantor set is universal in the class of compact metrizable spaces
in the sense of continuous onto mappings and even in a parametrized fashion
if dimension is taken into account.

Parovi\v{c}enko (\cite{MR0150732}) proved that, under $\CH$, every compact 
Hausdorff space of weight~$\cee$ is a continuous image of~$\Nstar$,
the remainder in the \v{C}ech-Stone compactification of~$\N$.
This all suggests that, still under~$\CH$, the space~$\Nstar$ should also
have this parametrized universality property.

The next result shows that this is not the case.
The space~$E2^\omega$, mentioned in the following proposition is the
\emph{absolute} or \emph{projective cover} of the Cantor set~$2^\omega$ ---
this is the unique (up to homeomorphism) extremally disconnected compact
space that admits a perfect irreducible map onto~$2^\omega$, 
see~\cite{MR1039321}*{Problem~6.3.19}.

\begin{proposition}\label{clopensubspace}
Let $X$ be a compact $F$-space that admits an at most $2$-to-$1$ map 
onto~$E2^\omega$. 
Then $X$ has a nonempty clopen subset that is homeomorphic to~$E2^\omega$.
\end{proposition}

\begin{proof}
Let $f:X\to E2^\omega$ be a continuous surjection whose fibers have size
at most~$2$.
There is a closed subspace $A$ of $X$ such that the restriction 
$f\restr A\to E2^\omega$ is irreducible.
Since $E2^\omega$ is extremally disconnected, it follows that 
$f\restr A$ is a homeomorphism (\cite{MR1039321}*{Problem~6.3.19\,(c)}). 
Hence we may as well assume that $E2^\omega$~is a subspace of~$X$ and that 
$f$~is a retraction from~$X$ onto~$E2^\omega$.  
We claim $E2^\omega$~has nonempty interior in~$X$, which clearly suffices.
Striving for a contradiction, assume that $E2^\omega$~is nowhere dense in~$X$. 
Let $\{E_n : n <\omega\}$ be a $\pi$-base for $E2^\omega$ consisting of clopen 
sets. 
For every $n$, $f\preim[E_n]$ is a clopen subset of $X$ such that 
$f\preim[E_n]\cap E2^\omega=E_n$.
As we are assuming that $E2^\omega$~is nowhere dense in~$X$ we can find
nonempty open $F_\sigma$-subsets $U_n$, with disjoint closures and such that
for all $n$, 
$$
\cl U_{3n}\cup\cl U_{3n+1}\cup\cl U_{3n+2}\subseteq f\preim[E_n]\setminus E2^\omega.
\eqno(1)
$$
(As $f\preim[E_n]\setminus\bigcup_{i\le k}\cl U_i$ is always open and nonempty 
it is never a subset of~$E2^\omega$.)

Put $U=\bigcup_n U_{3n}$, $V=\bigcup_n U_{3n+1}$, and $W=\bigcup_n U_{3n+2}$,
respectively. 
Since $X$ is a compact $F$-space the closures $\cl U$, $\cl V$ and~$\cl W$ of 
these sets are pairwise disjoint. 
By (1), $f[\cl U] = f[\cl V] = f[\cl W]= E2^\omega$, 
which contradicts $f$ being at most $2$-to-$1$.
\end{proof}

It is well known that $\Nstar$~is not separable and that every nonempty clopen
subset of~$\Nstar$ is homeomorphic to~$\Nstar$ itself.
As the space~$E2^\omega$ is separable this implies that $\Nstar$ does not admit
an at most $2$-to-$1$ continuous map onto~$E2^\omega$.

\end{document}